\documentclass[12pt,letterpaper]{article}
\usepackage[latin1]{inputenc}
\usepackage{hyperref}
\usepackage{amsmath}
\usepackage{url}
\usepackage{amsfonts}
\usepackage{amssymb}
\usepackage{amsthm}
\usepackage{verbatim}
\usepackage{graphicx}
\usepackage{xcolor}
\usepackage{tikz,tkz-graph}

\usepackage{latexsym}
\title{Modified Linear Programming and\\ Class 0 Bounds for Graph Pebbling}
\author{Daniel W. Cranston\footnote{Department of Mathematics and Applied
Mathematics, Virginia Commonwealth University, Richmond, VA, 23284,
\texttt{dcranston@vcu.edu}; This author is partially supported by NSA Grant 98230-15-1-0013.} ~ Luke Postle\footnote{Department of Combinatorics
and Optimization, University of Waterloo, Waterloo, Ontario, Canada, N2L 3G1,
\texttt{lpostle@uwaterloo.ca}} ~ Chenxiao Xue \footnote{ Department of
Mathematics and Computer Science, Davidson College, Davidson, NC 28035,
\texttt{chxue@davidson.edu}} ~ Carl Yerger\footnote{Department of Mathematics,
Davidson College, Davidson, NC 28035, \texttt{cayerger@davidson.edu} \newline
}}
\newtheorem{thm}{Theorem}
\newtheorem{prop}{Proposition}
\newtheorem{lem}{Lemma}
\newtheorem{claim}{Claim}
\newtheorem{cor}{Corollary}
\newtheorem{obs}{Observation}
\theoremstyle{definition}
\newtheorem{example}{Example}

\newcommand{\norm}[1]{\| #1 \|}

 \DeclareMathOperator{\ch}{ch}

\begin{document}
\def\ep{\varepsilon}
\def\lr{\left(}
\def\lf{\lfloor}
\def\rf{\rfloor}
\newcommand\floor[1]{\lfloor #1 \rfloor}
\newcommand\Floor[1]{\left\lfloor #1 \right\rfloor}
\def\lc{\left\{}
\def\rc{\right\}}
\def\rr{\right)}
\def\etapn{{2^{n+1} \left({\frac{1-8^{-(k_n+1)}}{7}}\right)+\floor{\frac{\alpha_n}{2}}}}
\def\letapn{{2^{n+1} \left({(1-8^{-(k_n+1)})}/7\right)+\floor{\frac{\alpha_n}{2}}}}
\def\p{\mathbb P}
\def\v{\mathbb V}
\def\e{\mathbb E}
\def\l{\mathbb L}
\def\lg{{\rm lg}}

\newcommand{\ignore}[1]{}

\providecommand{\abs}[1]{\vert#1\vert}
\providecommand{\norm}[1]{\Vert#1\Vert}
\providecommand{\flooralpha}{\left\lfloor\frac{\alpha_n}{2}\right\rfloor}
\providecommand{\floor}[1]{\left\lfloor#1\right\rfloor} 

\renewcommand\Box{\scalebox{0.7}{\mbox{$\square$}}}

\maketitle
\begin{abstract}
Given a configuration of pebbles on the vertices of a connected
graph $G$, a \textit{pebbling move} removes two
pebbles from some vertex and places one pebble on an
adjacent vertex. The \textit{pebbling number} of a graph $G$ is the
smallest integer $k$ such that for each vertex $v$ and each
configuration of $k$ pebbles on $G$ there is a sequence of pebbling
moves that places at least one pebble on $v$.

First, we improve on
results of Hurlbert, who introduced a linear optimization technique
for graph pebbling. In particular, we use a different set of
weight functions, based on graphs more general than trees.  
We apply this new idea to some graphs from
Hurlbert's paper to give improved bounds on their pebbling numbers.

Second, we investigate the structure of Class 0 graphs with few edges.
We show that every $n$-vertex Class 0 graph has at least
$\frac53n - \frac{11}3$ edges. This disproves a conjecture of Blasiak et al.
For diameter 2 graphs, we strengthen this lower bound to $2n - 5$, which is best
possible.  Further, we characterize the graphs where the bound holds with
equality and extend the argument to obtain an identical bound for diameter 2
graphs with no cut-vertex. 
\end{abstract}

\section{Introduction}

Graph pebbling was introduced by Chung 
in 1989.
Following a suggestion of Lagarias and Saks, she computed the pebbling number
of Cartesian products of paths to give a combinatorial proof of the following
number-theoretic result of Kleitman and Lemke.
\begin{thm} \cite{chung,lemke}
Let $\mathbb{Z}_n$ be the cyclic group on $n$ elements and let $|g|$
denote the order of a group element $g \in \mathbb{Z}_n$.  For every
sequence $g_1, g_2, \ldots, g_n$ of (not necessarily distinct)
elements of $\mathbb{Z}_n$, there exists a zero-sum subsequence
$(g_k)_{k \in K}$, such that $\sum_{k \in K} \frac{1}{|g_k|} \leq
1$.  Here $K$ is the set of indices of the elements in the
subsequence.
\label{thm:kleitman}
\end{thm}
Chung developed the pebbling game to give a more natural proof of
this theorem. Results of this type are important in this
area of number theory, as they generalize zero-sum theorems such as
the Erd\H{o}s-Ginzburg-Ziv theorem \cite{egz}. Over the past two decades,
pebbling has developed into its own subfield \cite{hurlbert,glennsurvey2},
with over 80 papers.

We consider a connected graph $G$ with {\it pebbles}
(indistinguishable markers) on some of its vertices.
More precisely, a \emph{configuration} $p$ on a
graph $G$ is a function from $V(G)$ to $\mathbb{N} \cup \{0\}$.
The \emph{size of $p$}, denoted $|p|$, is $\sum_{v\in V(G)}p(v)$.  A
{\it pebbling move} removes two pebbles from some vertex and places one pebble
on an adjacent vertex.  A \emph{rooted graph} is a pair $(G,r)$ where
$G$ is a graph and $r \in V(G)$ is the \emph{root vertex}.  A
pebbling configuration $p$ is \emph{solvable} for a rooted graph $(G,
r)$ if some configuration $p'$ has at least one pebble on $r$, and $p'$ can
be obtained from $p$ by a sequence of pebbling moves.
Otherwise, $p$ is \emph{unsolvable} (or \textit{r-unsolvable}, when the root
$r$ is specified.)

The {\it pebbling number} $\pi(G)$ is the least integer
$k$ such that, for any vertex $v\in V(G)$ and any initial
configuration $p$ of $k$ pebbles, $p$ is solvable for $(G,v)$.
Likewise $\pi(G,r)$ is the pebbling number
of $G$, when the root vertex must be $r$.
A trivial lower bound for $\pi(G)$ is $|V(G)|$: for some root $r$, we
place one pebble on each vertex other than $r$, for a total of $|V(G)|-1$ pebbles,
but we cannot reach $r$.

The path, $P_n$, on $n$ vertices has $\pi(P_n)=2^{n-1}$.
More generally, if graph $G$ has diameter $d$, then $\pi(G)\ge 2^d$.
Let $f(n,d)$ denote the maximum pebbling number of an $n$-vertex
graph with diameter $d$.  Pachter, Snevily, and Voxman
\cite{pachter} proved that $f(n,2) = n+1$, and Clarke, Hochberg, and
Hurlbert \cite{clarke} classified all graphs $G$ of diameter 2 with
$\pi(G)=n+1$.  Bukh \cite{bukh} proved
that $f(n,3) = 3n/2 + O(1)$, and Postle, Streib, and Yerger \cite{diam34} strenthened
Bukh's result, proving the exact bound $f(n,3) = \floor{3n/2} + 2$.
They also gave~\cite{diam34} an asymptotic bound for $f(n,4)$.

Section~\ref{sec:prelims} gives some necessary preliminaries; in particular 
it describes a
technique of Hurlbert \cite{hurllinprog} based on linear programming.  In Section
\ref{sec:weights}, we improve Hurlbert's method by using a different set of
weight functions, based on graphs more general than trees.  We apply this new
idea to some graphs from his paper to give improved bounds on their pebbling
numbers.

In Section~\ref{sec:class0}, we investigate the structure of Class 0 graphs
(graphs $G$ with $\pi(G)=|V(G)|$) with few edges.  We show that every
$n$-vertex Class 0 graph has at least $\frac53n - \frac{11}3$ edges.
This disproves a conjecture of Blasiak et al~\cite{blasiak}.
For diameter 2 graphs, we strengthen this bound to $2n - 5$ edges, which is
best possible.  We characterize the graphs where it holds with equality and
extend the argument to obtain an identical bound for diameter 2 graphs with
no cut-vertex.  

\section{Linear Programming Preliminaries}
\label{sec:prelims}

Computing a graph's pebbling number is hard.
Watson \cite{watson} and Clark and Milans \cite{milans}
studied the complexity of graph pebbling and some of its
variants, including optimal pebbling and cover pebbling.
Watson showed that it is NP-complete to determine whether a given configuration
is solvable for a given rooted graph $(G,r)$.
Clark and Milans refined this result, showing that deciding whether
$\pi(G) \leq k$ is $\Pi_2^P$-complete; this means
that pebbling is in the class of problems computable in polynomial
time by a co-NP machine equipped with an oracle for an NP-Complete
language.

Hurlbert \cite{hurllinprog} introduced a new linear programming technique, in
hopes of more efficiently computing bounds on pebbling numbers.
Before we describe our improvements on it, we briefly explain his method. Let
$G$ be a graph and let $T$ be a subtree of $G$ rooted at $r$.  For each $v\in
V(T)-r$, let $v^+$ be the parent of $v$, the neighbor of $v$ in $T$ that is
closer to $r$.  A \textit{tree strategy} is a tree $T$ and an associated
nonnegative \textit{weight function} $w_T$ (or $w$ if the context is
clear) where $w(r) = 0$ and $w(v^+) = 2w(v)$ for every vertex
not adjacent to $r$.  Further, $w(v) = 0$ if $v \not \in V(T)$.
Let $\mathbf{1}_G$ be the vector on $V(G)$ in which every entry is 1.

Hurlbert \cite{hurllinprog} proposed a general method for defining such a weight
function through tree strategies.  He proved the following
result (here $\cdot$ denotes dot product).

\begin{lem} \label{lem:wtfunction}  Let $T$ be a tree strategy of $G$ rooted at $r$,
with associated weight function $w$.  If $p$ is an $r$-unsolvable
configuration of pebbles on $V(G)$, then $w \cdot p \leq w \cdot
\mathbf{1}_G$. \end{lem}

The proof idea is easy.  Suppose that $p$ is a configuration with $w\cdot p >
w\cdot \mathbf{1}_G$.  This implies that some vertex $v$ in $T$ has at least
two pebbles.  Now we make a pebbling move from $v$ toward the root, i.e., from
$v$ to $v^+$, to get a new configuration $p'$.  Since $w(v^+)=2w(v)$, we have
$w\cdot p' = w\cdot p > w\cdot \mathbf{1}_G$.  By repeating this process, we
can eventually move a pebble to the root, $r$.

Since every $r$-unsolvable pebbling configuration $p$ satisfies $w \cdot p
\leq w \cdot \mathbf{1}_G$,
it follows that $\pi(G,r)$ is bounded above by one plus the number of pebbles
in the largest configuration $p$ such that $w\cdot p\le w\cdot \mathbf{1}_G$.
Let $\mathcal{T}_r$ be the set of all tree strategies in $G$ associated with
root vertex $r$.
By applying Lemma~\ref{lem:wtfunction} to all of $\mathcal{T}_r$ simultaneously,
we arrive at the following integer linear program: 

$$\begin{array}{rl}
   \max     & \sum_{v \neq r} p(v) \\
    \mbox{s.t.} & w \cdot p \leq w \cdot \mathbf{1}_G \\
                &  \mbox{for all }T \in \mathcal{T}_r.
    \end{array}
$$

Let $z_{G,r}$ be the optimal value of this integer linear program and let
$\hat{z}_{G,r}$ be the optimum of the linear relaxation, so that
configurations can be rational.  Since $z_{G,r}\le \floor{\hat{z}_{G,r}}$, we get the
bound $\pi(G,r) \leq z_{G,r} + 1\le \floor{\hat{z}_{G,r}}+1$.
Let $w_1,\ldots,w_k$ be weight functions of tree strategies for trees (possibly
different) rooted at $r$, and let $w'$ be a convex combination of
$w_1,\ldots,w_k$.  If $p$ is an $r$-unsolvable configuration, then $w'\cdot p\le
w'\cdot \mathbf{1}_G$ (otherwise $w_i\cdot p > w_i\cdot\mathbf{1}_G$, for some
$i$, a contradiction).  Further, if $w'(v)\ge 1$ for all $v$, then $|p|\le
\sum_{v\ne r}\floor{w'(v)}p(v)\le\floor{w'}\cdot\mathbf{1}_G$.  For ease of
application, we state this observation in a slightly more general form. 
We call this the Covering Lemma.

\begin{lem}[Covering Lemma]
For a graph $G$ and a root $r\in V(G)$, let $w'$ be a convex combination of tree
strategies for $r$, and let $C$ and $M$ be positive constants.  If $w'(v)\ge C$
for all $v\in V(G)\setminus\{r\}$ and 
$\sum_{v\in V(G)\setminus \{r\}}w'(v)< M$, then $\pi(G,r)\le
\Floor{\frac{M}C}+1$.
In particular, if
$\sum_{v\in V(G)\setminus \{r\}}w'(v)< C|V(G)|$, then
$\pi(G,r)=|V(G)|$.  \label{coveringLemma}
\end{lem}
\noindent
For any bound on $\pi(G)$ arising from such a $w'$, 
a \emph{certificate} of the
bound consists of 
the strategies $w_i$ and their coefficents in the
convex combination forming $w'$.


Hurlbert applies this linear programming method more broadly by
considering strategies on trees where $w(v^+) \geq 2w(v)$,
called \textit{nonbasic strategies}.
Since nonbasic strategies are conic combinations of basic strategies
\cite[Lemma 5]{hurllinprog}, this extension does not strengthen the
method.  However, it often yields simpler certificates.

\section{More General Weight Functions}
\label{sec:weights}

Here we generalize the notion of weight function from
the previous section to allow
weight functions for graphs $G$ that are not trees.
A \emph{weight function} is a map $w: V(G)\rightarrow
\mathbf{R}^+\cup\{0\}$.  A weight function for a graph $G$ and root $r$ is
\emph{valid} if $w(r)=0$ and every $r$-unsolvable configuration $p$ satisfies
$w\cdot p\le w\cdot\mathbf{1}_G$.
Although it is harder to show that one of these more general weight
functions is valid,
when we can, this often leads to improved pebbling
bounds for a variety of graph families.
Given a graph $G$ and a root $r$, it is straightforward to check that the
theory developed in the previous section extends to any weight function $w$
such that $w\cdot p \le w\cdot \mathbf{1}_G$ for every configuration $p$ that
is not $r$-solvable.  Our next result establishes a new family of such weight
functions.  A \emph{$k$-vertex} is a vertex of degree $k$.

\begin{lem}
Form $G$ from an even cycle $C_{2t}$ by identifying one vertex with the
endpoint of a path of length $s-t$.  Let $x_t$ be the resulting 3-vertex and $x_0$
be the 2-vertex farthest from $x_t$; now $x_0$ and $x_t$ split the even cycle
into two paths, $P_1$ and $P_2$.  Label the internal vertices of $P_1$ as $x'_1,
x'_2, \ldots, x'_{t-2},x'_{t-1}$ and the internal vertices of $P_2$ as $x''_1,
x''_2, \ldots, x''_{t-2}, x''_{t-1}$.  Call the 1-vertex $r$, and let $P_3$ be
the path from $x_t$ to $r$.
Label the internal vertices of $P_3$ as $x_{t+1}, x_{t+2},\ldots$, $x_{s-1}, x_s$.
For each $i\ne 0$, give weight $2^i$ to vertex $x_i$ or vertices $x'_i$ and
$x''_i$.  Let $\alpha = \frac{2^s+2^{t-1}-2}{2^s-1}$ and give weight $\alpha$ to
$x_0$.  Fix some order on the vertices, and let $w$ be the vector of length
$|V(G)|$ where entry $i$ is the weight of vertex $i$.  If $p$ is an
$r$-unsolvable configuration, then $w\cdot p \le w\cdot \mathbf{1}_G$.
\label{lem:cycleTail}
\end{lem}

\begin{proof}
Figures~\ref{fig:square} and~\ref{fig:bruhatweight} both show examples of this
lemma, which we apply later.

Let $p$ be an $r$-unsolvable configuration.  We will show that $w\cdot p\le
w\cdot\mathbf{1}_G$.
Let $M = \alpha + 2^{s+1}+2^t-4$.  Note that $w\cdot \mathbf{1}_G = M$.  
Let $W_0=\alpha p(x_0)$, 
$W_L=\sum_{v\in P_1\setminus\{x_0,x_t\}}w(v) p(v)$,
$W_R=\sum_{v\in P_2\setminus\{x_0,x_t\}}w(v) p(v)$,
and
$W_C=\sum_{v\in P_3\setminus\{r\}}w(v) p(v)$. (Here $L$, $R$, and $C$ stand for
left, right, and center.)
We will show that $W_0+W_L+W_R+W_C \le M$.

\begin{claim}
If $W_L=0$ or $W_R=0$, then the lemma is true.
\end{claim}

By symmetry, assume that $W_L=0$.  Now Lemma~\ref{lem:wtfunction} implies that
$\alpha^{-1}W_0+W_R+W_C \le 2^{S+1}-1$.  Multiplying by
$\alpha$ gives $W_0+W_L+W_R+W_C\le W_0+W_L+\alpha(W_R+W_C)\le \alpha (2^{s+1}-1)
= \alpha+\alpha(2^{s+1}-2) = \alpha+2(2^s+2^{t-1}-2)=\alpha+2^{s+1}+2^t-4=M$.
This proves the claim.

\begin{claim}
If $W_L+\alpha^{-1}W_0 > 2^t-1$, then the lemma is true.
\end{claim}

Suppose that $W_L+\alpha^{-1}W_0 > 2^t-1$.  
We can assume that $W_L<2^t$; otherwise we can move weight down to $x_t$,
without changing the sum $W_L+W_C$.  Now we move some pebbles toward the root
and reduce to the case in Claim~1.  Specifically, we remove $2^t-W_L$ pebbles from
$x_0$ and place half that many on $x'_1$.  Call the new configuration $p'$ and
define $W'_0$, $W'_L$, $W'_R$, and $W'_C$, analogously.
Note that $\alpha^{-1}W'_0+W'_L+W'_R+W'_C=\alpha^{-1}W_0+W_L+W_R+W_C$. 
Since $W'_L=2^t$, we can move all weight from internal vertices of $P_1$ to
$x_t$.  This gives a new configuration $p''$ with $W''_L=0$.
Again $\alpha^{-1}W''_0+W''_L+W''_R+W''_C=\alpha^{-1}W_0+W_L+W_R+W_C$. So the
claim holds by Claim~1.
\smallskip

By Claim~2, we now assume that $W_L+\alpha^{-1}W_0 \le 2^t-1$.
By symmetry, we assume that $W_R+\alpha^{-1}W_0 \le 2^t-1$.
By Lemma~\ref{lem:wtfunction}, we can also assume that $W_C\le 2^{s+1}-2^{t}$. 
Adding these inequalities yields 
\begin{align}
W_C+W_L+W_R+2\alpha^{-1}W_0 &\le 2(2^t-1)+2^{s+1}-2^t\nonumber\\
& =2^{s+1}+s^t-2.
\end{align}
We can assume that $W_0>0$, since otherwise the lemma holds by
Lemma~\ref{lem:wtfunction}.  Thus,
we have $W_0\ge \alpha$, so $(2\alpha^{-1}-1)W_0 \ge 2-\alpha$.  Subtracting
this inequality from (1) gives the desired result.
\end{proof}

The following observation extends our class of valid weight functions a bit
further.

\begin{obs}
Let $G$ be a graph and $r$ a root; let $w$ be a weight function on $G$ such that
$w\cdot p\le w\cdot \mathbf{1}_G$ for every $r$-unsolvable configuration $p$.
Form $G'$ from $G$ by adding a new vertex $u$ adjacent to some vertex $u^+$ of
$G$ (with $u^+\ne r$), and form $w'$ from $w$, where $w'(u)=\frac12w(u^+)$ and
$w'(v)=w(v)$ for every $v\in V(G)$.  For every $r$-unsolvable configuration $p'$
in $G'$, we have $w'\cdot p'\le w'\cdot\mathbf{1}_{G'}$.  Further, we can allow,
more generally, that $w'(u) \le \frac12w(u^+)$.  We can also attach trees,
rather than single vertices.
\label{obs1}
\end{obs}
\begin{proof}
If the new vertex $u$ has more than one pebble, we move as much weight as
possible from $u$ to $u^+$, which does not decrease the total weight on $G$. 
This proves the first statement.  The second
statement follows from taking convex combinations of $w$ and $w'$.  The final
statement follows by induction on the size of the tree $T$ that we attach (we
just proved the induction step, and the base case, $|T|=0$, is trivial).
\end{proof}

\subsection{The Cube and the Lemke Graph}

To illustrate the usefulness of Lemma~\ref{lem:cycleTail} and
Observation~\ref{obs1}, we give two easy applications of this method. We show
that $\pi(Q_3) = 8$ and $\pi(L) = 8$, where $Q_3$ is the 3-dimensional cube and
$L$ is the Lemke graph, shown in Figure~\ref{fig:lemkeg}.  When using tree
strategies alone, Hurlbert's method cannot handle these graphs.

\begin{figure}[htb]
\begin{center}
\includegraphics[scale=.6]{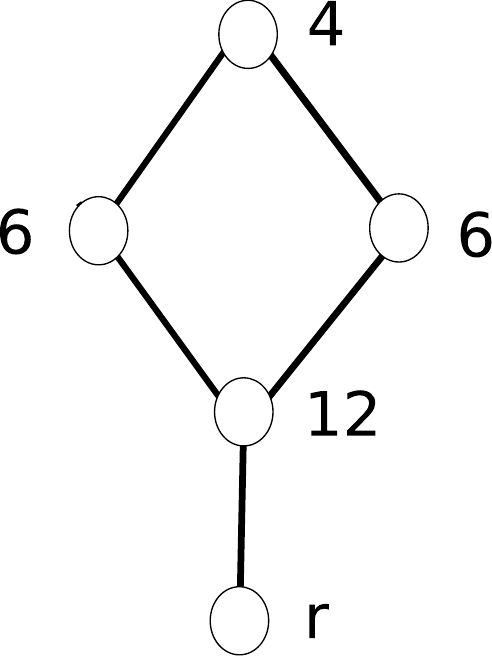}
\unitlength 1.2mm 
\caption{A valid non-tree weight function.\label{fig:square}}
\end{center}
\end{figure}

\begin{prop}
If $Q_3$ is the 3-cube, then $\pi(Q_3) = 8$.
\end{prop}

\begin{proof}
Every graph $G$ satisfies $\pi(G)\ge |V(G)|$, so $\pi(Q_3)\ge 8$.
Thus, we focus on proving that $\pi(Q_3)\le 8$.

To show that $\pi(Q_3) \le 8$, we first note that the weight function in Figure
\ref{fig:square} is valid.  Since a valid weight function remains valid when
multiplied by a positive constant (in this case 3), this statement follows from
Lemma~\ref{lem:cycleTail}, with $t=2$ and $s=0$.

%

\begin{figure}[bht]
\begin{center}
\includegraphics[scale=.6]{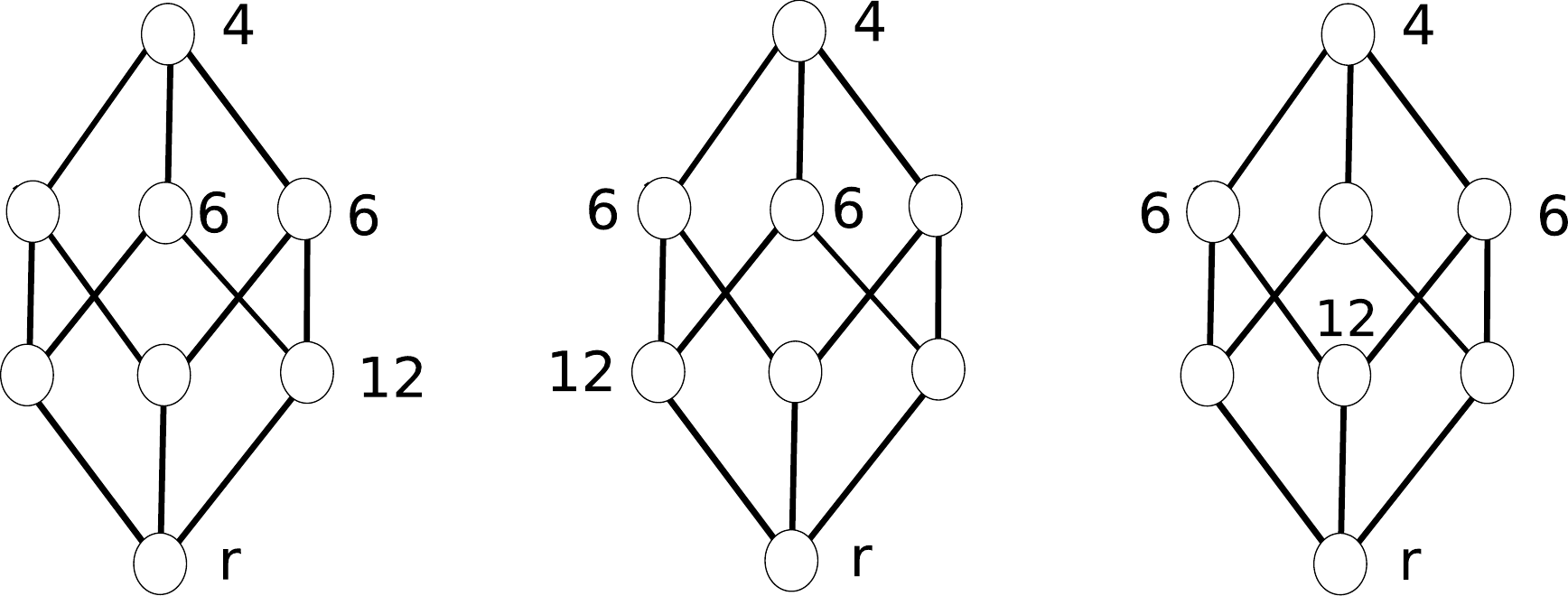}
\unitlength 1.2mm \caption{A certificate that $\pi(Q_3) \leq 8$.}
\label{cubeconfig}
\end{center}
\end{figure}

The convex combination of the three strategies shown in Figure
\ref{cubeconfig} (each taken with weight 1) yields $w'$ such that $w'(v)=12$
for all $v\ne r$.  Thus, the 
\hyperref[coveringLemma]{Covering Lemma} 
shows that $\pi(Q_3)\le 8$,
so $Q_3$ is Class 0.  These three strategies in Figure \ref{cubeconfig} also
serve as a certificate that $\pi(Q_3) \leq 8$, and they yield an efficient
algorithm for getting a pebble to $r$, starting from any configuration $p$ with
$|p|\ge 8$.
\end{proof}

The most famous long-standing pebbling problem is Graham's
conjecture: for all graphs $G_1$ and $G_2$,  $\pi(G_1 \Box G_2) \leq
\pi(G_1)\pi(G_2)$; here $\Box$ denotes the Cartesian product.  This
conjecture has been verified only for a few classes
of graphs.  Specifically, it holds when $G_1$ and $G_2$ are both cycles
\cite{Her1}, both trees \cite{moews}, both complete bipartite graphs
\cite{feng2}, or a fan and a wheel \cite{feng}.

When considering Graham's conjecture, we are interested in the Lemke Graph,
denoted $L$ and shown on the left in Figure \ref{fig:lemkeg}.  This graph is of interest
because it is the smallest graph without the \textit{2-pebbling property}.
The exact definition is unimportant for us here; what matters, is that if $G$
has this property, then $\pi(G\Box H)\le \pi(G)\pi(H)$ for every graph $H$.
This makes $L \Box L$ a natural candidate for disproving Graham's
conjecture.
\begin{figure}[htb]
\begin{center}
\includegraphics[scale=.75]{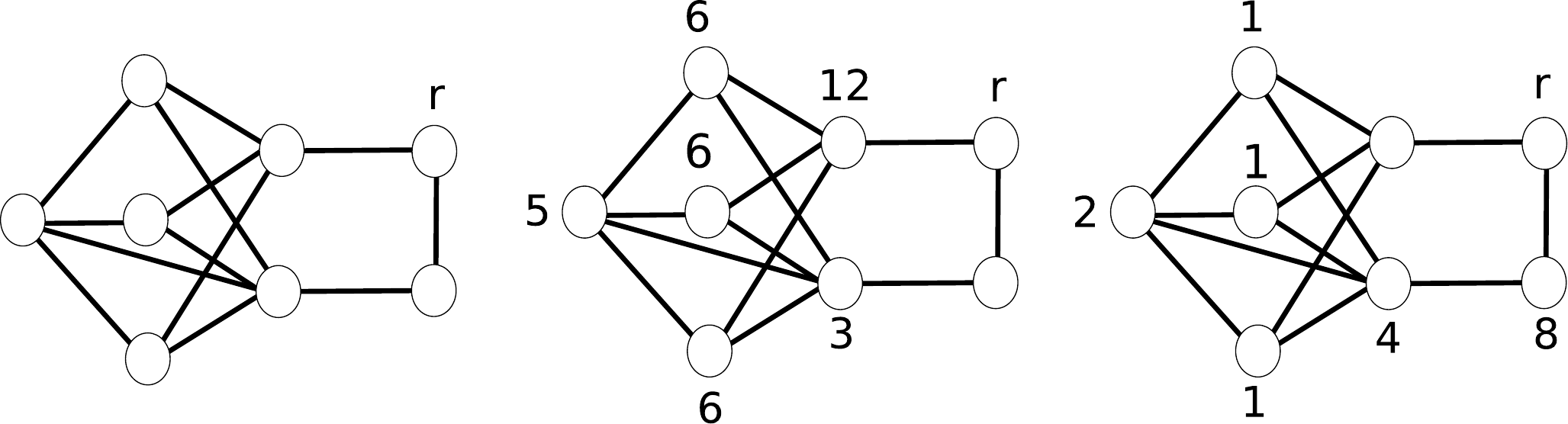}
\unitlength 1.2mm \caption{The Lemke graph.\label{fig:lemkeg}}
\end{center}
\end{figure}

Hurlbert asserted that it is impossible, using tree strategies alone, to
obtain the pebbling number of the Lemke graph via the linear programming
technique. However, by using this method with more general weight functions, we
prove that $\pi(L) = 8$.

\ignore{Before we do this, we would like to prove a lemma regarding
adding subtrees to weight functions that satisfy the constraints for
a valid weight function.  In particular we prove a lemma that allows
us to attach trees that fit the constraints of a \textit{strategy}
within the context of Hurlbert's paper.  Need to state and prove a
theorem here, should be about 1/2 a page.

Idea:  Use induction on the number of vertices of all the added
trees.  Use induction ala Hurlbert's paper (basically make the same
argument on the trees).}

\begin{figure}[!bht]
\begin{center}
\includegraphics[scale=.7]{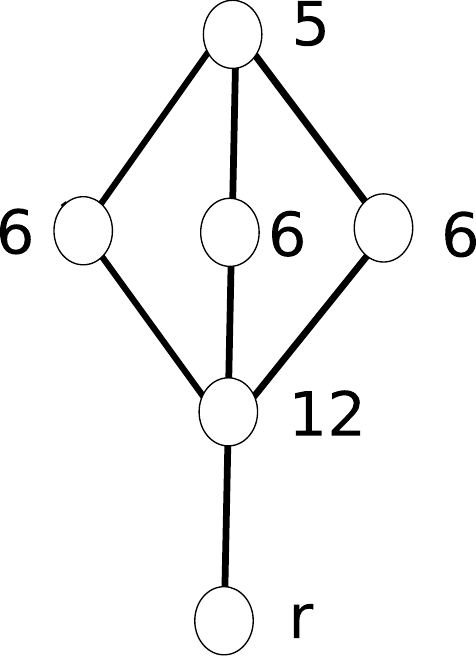}
\unitlength 1.2mm \caption{A weight function useful for the Lemke
graph.\label{fig:threepth}} 
\end{center}
\end{figure}

\begin{thm}
If $L$ is the Lemke graph, then $\pi(L) = 8$.
\end{thm}

\begin{proof}
Note that $\pi(L) \geq |V(L)|=8$, so we focus on proving the upper bound. 
Hurlbert~\cite{hurllinprog} showed that $\pi(L,v) = 8$ for all
vertices $v \in L$ except for $r$, as shown on the left in
Figure~\ref{fig:lemkeg}.  So we only need to show that
$\pi(L, r) = 8$.  Now we need the weight function in
Figure~\ref{fig:threepth}.

\setcounter{claim}{0}
\begin{claim} \label{wtfun2}
The weight function in Figure \ref{fig:threepth} is valid.
\end{claim}

The proof of this claim is very similar to the proof of Observation~\ref{obs1},
so we just sketch the ideas.  If any vertex weighted 6 has no pebbles, then we
invoke the weight function in Figure~\ref{fig:square}, and
multiply the resulting inequality by $\frac54$ to get one that implies what we
want; so we assume that each vertex weighted 6 has a pebble. If the vertex
weighted 12 has a pebble, then the vertex weighted 5 has
at most one pebble, so we are done.  Otherwise, the vertex weighted 5 has at
most 3 pebbles; again, we are done.  This proves the claim.
%

The proof that $\pi(L,r)\le 8$ uses the two strategies in
Figure~\ref{fig:lemkeg}.  The rightmost is a nonbasic tree strategy. The center
strategy is derived from the weight function in Figure~\ref{fig:threepth} by
adding a vertex with weight 3 adjacent to some vertex with weight 6.  This
weight function is valid, by Observation~\ref{obs1}.  When we sum the weights
of the two strategies, each vertex has weight at least 7 and the total weight
is 55.  Now the \hyperref[coveringLemma]{Covering Lemma} implies that
$\pi(L,r)\le \floor{\frac{55}7}+1=8$.
\end{proof}

\subsection{Larger Graphs}

In this section we determine the
pebbling number of the Bruhat graph of order 4. The (weak)
\textit{Bruhat graph} of order $m$ has as its vertices
the permutations of $\{1,\ldots,m\}$; two vertices are adjacent if the
corresponding permutations differ by an adjacent transposition.  Since this graph is
vertex-transitive, we can choose the root vertex arbitrarily.  
Using the linear
programming method, Hurlbert proved that $\pi(B_4)~\leq~72$.
By using more general weight functions, we calculate the pebbling
number of this graph exactly.  
\begin{figure}[bht]
\begin{center}
\includegraphics[width=14cm,height=9cm]{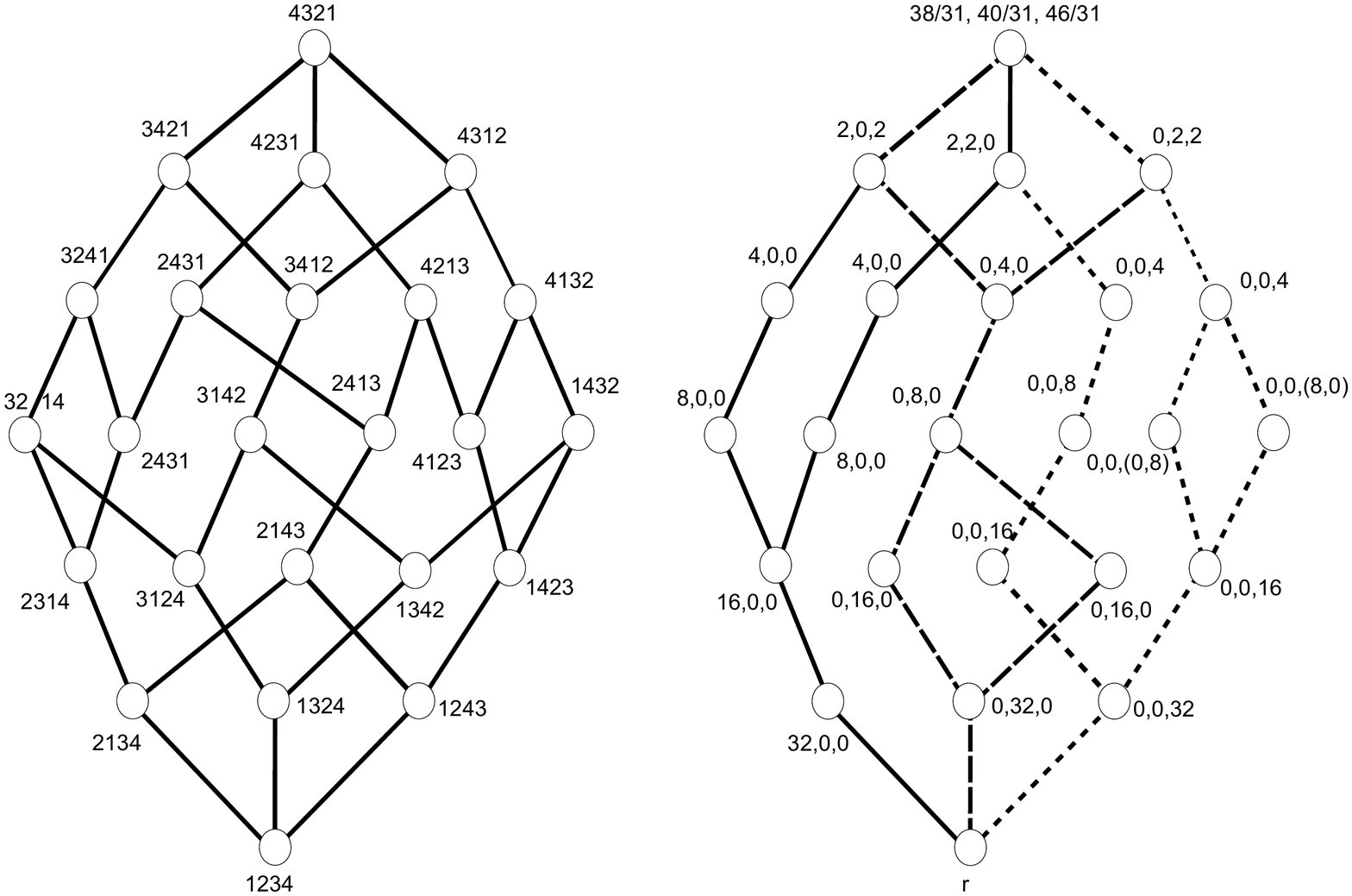}
\unitlength 1.2mm \caption{The Bruhat graph of order 4 and a set of
strategies proving $\pi(B_4) = 64$.} \label{fig:bruhat}
\end{center}
\end{figure}

\begin{thm}
If $B_4$ is the Bruhat graph of order 4, then $\pi(B_4) = 64$.
\end{thm}

\begin{proof}
The diameter of $B_4$ is 6, so $\pi(B_4) \geq
2^6=64$.  We need to show that $\pi(B_4) \leq 64$.

Note that the rightmost graph in Figure~\ref{fig:bruhatweight} describes two
strategies, as we explain below.  
We combine these four strategies, as shown in Figure~\ref{fig:bruhat},
(weighted with multiplicities $\frac14, \frac14,\frac18,\frac18$) to get a
weight function $w'$, such that $w'(v)\ge 1$ for all $v$ and $w'\cdot
\mathbf{1}_{B_4}=63$.  This proves the desired upper bound $\pi(B_4)\le 63+1$. 
Thus, we only need to show that the four strategies in
Figure~\ref{fig:bruhatweight} are valid.

\begin{figure}[!htb]
\begin{center}
\includegraphics[scale=.5]{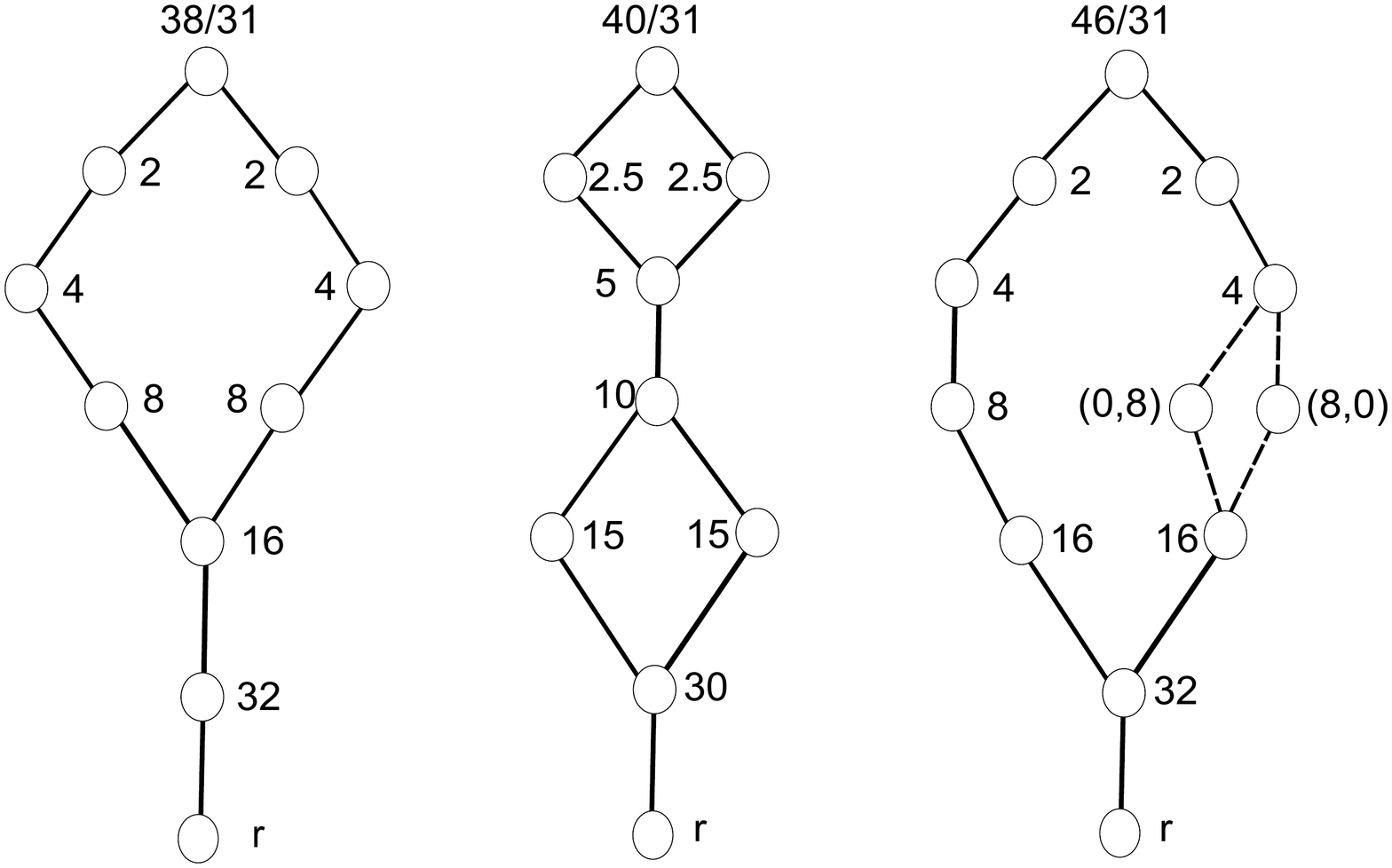}
\unitlength 1.2mm \caption{More general weight functions.
\label{fig:bruhatweight}}
\end{center}
\end{figure}

The weight function on the right denotes two different weight functions on $G$;
the first includes the vertex labeled $(8,0)$ but not the one labeled $(0,8)$,
and the second vice versa.  By Lemma~\ref{lem:cycleTail}, the leftmost and
rightmost strategies are valid (the former with $t=4$ and $s=1$; the latter
with $t=5$ and $s=0$).

The proof that the middle strategy is valid is similar to the proof of
Lemma~\ref{lem:cycleTail}, so we just sketch the ideas.  Note that weights 30,
15, 10, and 5 (with the other vertices unweighted) are consistent with
Lemma~\ref{lem:cycleTail} (mulitplied by $\frac{15}2$), when $t=2$ and $s=0$, and
adding a vertex by Observation~\ref{obs1}.  Thus, we know that if $p$ is
$r$-unsolvable, then these five vertices have weight at most 75.
The key observation is that these five vertice can play the role of $P_3$ in the
proof of Lemma~\ref{lem:cycleTail}.  Let $x_0$, $x'_1$, and $x''_1$ denote the
vertex labeled $\frac{40}{31}$ and its two neighbors, respectively.
We first consider the case where $x'_1$ or $x''_1$, say $x'_1$, has no pebbles.
In this case we move as much weight as possible to the vertex labeled 5 from
$x_0$ and $x''_1$.  We also consider the case where both $x'_1$ and $x''_1$ have
pebbles.  Either we can reduce to the previous case, or else we get
two inequalities.  We add these two to the inequality for the bottom 5 vertices,
which gives the desired inequality.
\end{proof}



\section{Class 0 Graphs}
\label{sec:class0}

\subsection{Preliminaries}

In this section, we study Class 0 graphs.
We focus on graphs with diameter at least 2 since those with diameter 0, a
single vertex, and diameter 1, a complete graph, are well understood.
A graph $G$ is \textit{Class 0} if its pebbling number is equal to its number
of vertices, i.e., $\pi(G)=|V(G)|$.  Recall that always $\pi(G)\ge |V(G)|$, so
Class 0 graphs are those where this trivial lower bound holds with equality.
For each vertex $v$, we write $N(v)$ for the set of vertices adjacent to $v$,
and we write $N[v]$ to denote $N(v)\cup\{v\}$.
For a graph $G$, let $e(G)$ denote the the number of edges in $G$.
In this section, we prove lower bounds on $e(G)$ for all Class 0 graphs.

Blasiak et al.~\cite{blasiak} showed that every $n$-vertex Class 0 graph $G$
has $e(G)\ge \Floor{\frac{3n}{2}}$.  They also conjectured (see
\cite[p.~19]{hurllinprog}) that for some constant $C$ and for all sufficiently
large $n$ there exist $n$-vertex Class 0 graphs with $e(G)\le \Floor{\frac{3n}{2}}+C$.
In particular, they defined a family of ``generalized Petersen graphs'' of
arbitrary size and diameter with one vertex of some fixed degree $m$ and all
other vertices of degree 3.  They conjectured that these graphs are all Class 0.
We disprove this conjecture in a very strong sense.  Shortly, we prove that for
fixed $m$, all sufficiently large graphs of this form are not Class 0. 
(Figure~\ref{fig:genPet} shows $P_{8,2}$, one of these generalized Petersen
graphs that is not Class 0.) Later in this section, we extend this idea to show
that every $n$-vertex Class 0 graph $G$ has $e(G)\ge \frac53n-\frac{11}3$.
To conclude the section, for all diameter 2
graphs $G$ we strengthen this lower bound to $e(G) \ge 2n - 5$.
Further, we characterize the graphs where this bound holds with
equality (which include two infinite families).
\begin{figure}[bht]
\centering
\includegraphics[scale = 1.1]{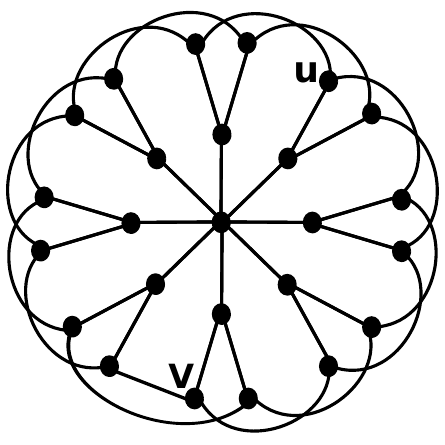}
\caption{The generalized Petersen graph, $P_{8,2}$, is not Class 0.}
\label{fig:genPet}
\end{figure}
%


Our main tool for proving bounds on $e(G)$ is the following lemma.

\begin{lem}[Small Neighborhood Lemma]
Let $G$ be a Class 0 graph.  If $u, v\in V(G)$, $d(u)=2$, and $u$ and $v$ are
distance at least 3 apart, then $d(v)\ge 4$.  Similarly, if $u,v\in V(G)$,
$d(u)=3$, $u$ and $v$ are distance at least 4 apart, and each neighbor of $v$
is a 3-vertex, then $d(v)\ge 4$.
\label{smallNeighborhood}
\end{lem}
\begin{proof}
The proofs for both statements are similar. In each case, we assume the statement
is false and construct a configuration with $|V(G)|$ vertices that is
$u$-unsolvable.  Consider the first statement first.  Suppose, to the contrary,
that $u$ and $v$ are as required, but $d(v)\le 3$.  Form configuration $p$ by
putting 7 pebbles on $v$, 0 pebbles on each vertex of $N[u]\cup N(v)$, and 1
pebble on each other vertex.  Since $|N[u]\cup N[v]|\le 7$, this configuration
has at least $|V(G)|$ pebbles.
Now no pebble can reach $u$, since at most one pebble can leave $N[r]$.
This contradicts that $G$ is Class 0, so $d(v)\ge 4$.

Now consider the second statement.
Suppose, to the contrary, that $u$ and $v$ are as required, but $d(v)\le 3$.
Form configuration $p$ by putting 15 pebbles on $v$, 0 pebbles on each vertex of
$N[u]\cup (N[N[v]]\setminus\{v\})$, and 1 pebble on each other vertex.  Since
$|N[u]\cup N[N[v]]|\le 15$, the configuration has at least $|V(G)|$ pebbles, but
no pebble can reach $v$, since at most one pebble can leave $N[N[v]]$.
This contradicts that $G$ is Class 0.  Thus, $d(v)\ge 4$.
\end{proof}

\begin{cor}
For each integer $C$, there exists an integer $n_0$, such that if
$G$ is any $n$-vertex graph with $\delta(G)=3$, $n\ge n_0$, and $e(G)\le
\frac32n+C$, then $G$ is not Class 0.  
\end{cor}
\begin{proof}
We can choose $n_0$ sufficiently large so that there exists some pair of
vertices $u$, $v$ violating the second statement of the 
\hyperref[smallNeighborhood]{Small Neighborhood Lemma}.
Specifically, it suffices to find a 3-vertex $v$ such that every vertex
within distance four of $v$ is a 3-vertex.  To guarantee such a vertex $v$, we
can take, for example, $n_0 = 2C*3^5$.
\end{proof}

\subsection{Diameter at least 3}
Now we use the \hyperref[smallNeighborhood]{Small Neighborhood Lemma}
to prove, in Theorem~\ref{thm:delta3},
that every $n$-vertex Class 0 graph $G$ with diameter at least 3 has $e(G)\ge
\frac53n-\frac{11}3$.  The case $\delta(G)=2$ is complicated, so we handle it
separately, in Lemma~\ref{lem:delta2}.  For the case $\delta(G)\le 1$, we use
the following easy lemma from~\cite{clarke}.

\begin{lem}[\cite{clarke}]
Every Class 0 graph $G$ has no cut-vertices.  Specifically, $\delta(G)\ge 2$.
\label{lem:NoCut}
\end{lem}

\begin{proof}
Let $G$ be a graph with a cut-vertex $u$ and neighbors $v_1$ and $v_2$ that are
in different components of $G-u$.  Consider the distribution $p$ with 3 pebbles
on $v_1$, 0 pebbles on each of $u$ and $v_2$, and 1 pebble on each other vertex.
Distribution $p$ has $|V(G)|$ pebbles, but no pebble can reach $v_2$, which we
now show.  If a pebble ever moves to $u$, then at that point each vertex has at
most one pebble, and $v_2$ has no pebbles.  Otherwise, every pebbling move is
within the component of $G-u$ containing $v_1$, so no pebble reaches $v_2$.
Thus no pebble can reach $v_2$, so $G$ is not Class 0.
\end{proof}

\begin{lem}
\label{lem:delta2}
If an $n$-vertex Class 0 graph $G$ has diameter at least 3 and $\delta(G) = 2$, then
$e(G)\ge \frac53n-\frac{11}3$.
\end{lem}

\begin{proof}
Let $G$ be an $n$-vertex Class 0 graph with diameter at least 3 and
$\delta(G)=2$.
We assign each vertex $v$ a charge $\ch(v)$, where $\ch(v)=d(v)$.
Now we redistribute these charges, without changing their sum, so that all but a few
vertices finish with charge at least $\frac{10}3$; the charge of each vertex
$v$ after redistributing is $\ch^*(v)$. If at most $k$ vertices finish with
charge less than $\frac{10}3$ (but all charges are nonnegative), then $e(G) =
\frac12\sum_{v\in V}\ch(v) =
\frac12\sum_{v\in V}\ch^*(v) \ge \frac12(\frac{10}3(n-k)) = \frac53n-\frac53k$.

%

Choose $r\in V(G)$ such that $d(r)=2$.  For each positive integer $i$, let $N_i$
denote the set of vertices at distance $i$ from $r$.
Also, let $N_{3^+}=\bigcup_{i\ge 3}N_i$.
By the \hyperref[smallNeighborhood]{Small Neighborhood Lemma}
with $u=r$, if $v\in N_{3^+}$, then $d(v)\ge 4$.

We redistribute charge according to the following two discharging rules.
\begin{enumerate}
\item Each vertex $v\in N_2$ takes charge 1 from some neighbor in $N_1$.
If $d(v)=2$, then $v$ also takes charge $\frac13$ from its other neighbor.
\item Each vertex $v\in N_{3^+}$ with $d(v)=4$ takes charge $\frac13$ from each
neighbor $u$ with $d(u)\ge 3$.
\end{enumerate}

We show that nearly all vertices finish with charge at least $\frac{10}3$.
Consider a vertex $v\in V(G)\setminus N[r]$.  If $d(v)\ge 5$, then $\ch^*(v)\ge
d(v)-\frac13d(v)=\frac23d(v) \ge \frac{10}3$.  Now suppose $v\in N_2$ and
$d(v)\ge 3$.  In this case, $\ch^*(v)\ge d(v)+1-\frac13(d(v)-1) =
\frac23d(v)+\frac43 \ge \frac{10}3$.  Suppose instead that $v\in N_2$,
$d(v)=2$, and either $v$ has both neighbors in $N_1$ or the neighbor of $v$ outside
of $N_1$ has degree at least 3.  Now $\ch^*(v)= d(v)+\frac43=\frac{10}3$.

We show that $G$ has at most two 2-vertices in $N_2$ with 2-neighbors in $N_2$.
 Suppose, to the contrary, that $u_1$, $u_2$, and $u_3$ are 2-vertices in
$N_2$, each with a 2-neighbor in $N_2$; by symmetry, assume
$u_1u_2\in E(G)$.
By Lemma~\ref{lem:NoCut}, $u_1$ and $u_2$ cannot have a common neighbor $v\in N_1$,
since then $v$ would be a cut-vertex.  Thus, $u_1$ and $u_2$ have distinct
neighbors in $N_1$.  However, now $u_3$ is distance three from either $u_1$ or $u_2$;
by symmetry, say $u_1$.  Now $u_1$ and $u_3$ contradict the 
\hyperref[smallNeighborhood]{Small Neighborhood Lemma}.
So indeed $N_2$ has at most two 2-vertices with 2-neighbors in $N_2$.

Now we consider 4-vertices in $N_{3^+}$.  Rather than compute the charges of
these 4-vertices individually, we group them together as follows.  Let $H$ be
the subgraph induced by 4-vertices in $N_{3^+}$, and let $H_1$ be a component
of $H$ with $k$ vertices.  If $H_1$ contains a cycle, then $H_1$ contains at
least $k$ edges, so vertices of $H_1$ give charge to at most $4k-2(k)=2k$
vertices outside $H_1$.  Thus, $\ch^*(H_1)\ge \ch(H_1)-2k(\frac13)=4k-\frac{2k}3
= \frac{10}3k$.  Similarly, if $H_1$ has some adjacent vertex that is not a
2-vertex, then $\ch^*(H_1)\ge \ch(H_1)-(2k+1)(\frac13)+\frac13 = \frac{10}3k$. 
Instead, assume that $H_1$ is a tree and every vertex adjacent to $H_1$ is a
2-vertex.  Recall that each such 2-vertex is in $N_2$.

If every 2-neighbor of $H$ is adjacent to the
same vertex of $N_1$, call it $v$, then $v$ is a cut-vertex.  Thus, $H_1$ has
2-neighbors that are adjacent to both vertices of $N_1$; call these 2-neighbors
$u_1$ and $u_2$.  By the \hyperref[smallNeighborhood]{Small Neighborhood Lemma},
 every pair of 2-vertices in
$N_2$ are adjacent or have a common neighbor.  Since $u_1$ and $u_2$ are both
adjacent to $H_1$, they can't be adjacent to each other; thus, they must have a
common neighbor, $u_3$.  Further, every 2-vertex in $N_2$ must be adjacent to
$u_3$.  Since $u_3\in V(H_1)$, $u_3$ is a 4-vertex, so $N_2$ has at most four
2-vertices.  Thus, $H_1$ is the only component of $H$ with final charge less
than $\frac{10}3$ times its size.  Furthermore, $H_1$ has only a single vertex,
and $\ch^*(H_1)=4-4(\frac13)=\frac83$. 

Now we compute the total final charge of $V(G)$.  For each vertex $v$ not in $H$, the
final \emph{excess} of $v$ is $\ch^*(v)-\frac{10}3$.  For each component $H_i$ of $H$
with order $k$, the final excess is $\ch^*(H_i)-\frac{10}3k$.  We now show
that the sum of all final excesses is greater than or equal to $-\frac{22}3$,
which proves the lemma.

If $v\in N_{3^+}$ and $d(v)\ge 5$, then $\ch^*(v)\ge \frac{10}3$, so $v$
has nonnegative excess.  Each component of $H$, other than (possibly) $H_1$,
has nonnegative excess.  Further,
$H_1$ has excess
greater than or equal to $-\frac23$.  Each $v\in N_2$ with $d(v)\ge 3$ has
nonnegative excess.
Also, each $v\in N_2$ with $d(v)=2$ has excess 0, except for
at most two adjacent 2-vertices, which each have excess $-\frac13$.
Finally, the sum of the final charges on $N[r]$ is at least $4$
(since $r$ takes no charge from $N(r)$).  Thus, the sum of excesses of
$N[r]$ is at least $4-3(\frac{10}3)=-6$.  So the sum of excesses over all
vertices is at least $2(-\frac13)+(-\frac23)+(-6)=-\frac{22}3$.  Thus $\sum_{v\in
V(G)}d(v)\ge \frac{10}3n-\frac{22}3$, so $e(G)\ge \frac53n-\frac{11}3$.
\end{proof}

Now we prove our main theorem of this section.

\begin{thm}
\label{thm:delta3}
If $G$ is an $n$-vertex Class 0 graph with diameter at least 3, then $e(G)\ge
\frac53n-\frac{11}3$.
\end{thm}
\begin{proof}
Let $G$ be Class 0 with diameter at least 3.  By Lemma~\ref{lem:NoCut},
$\delta(G)\ge 2$.  Lemma~\ref{lem:delta2} proves the bound when $\delta(G)=2$.
If $\delta(G)\ge 4$, then $e(G) \ge \frac{\delta(G)n}2\ge 2n$.  Thus, we assume
that $\delta(G)=3$.

The proof is similar to that of Lemma~\ref{lem:delta2}, but easier.
Recall that a \emph{$k$-vertex} is a vertex of degree $k$.  Similarly, a
\emph{$k^+$-vertex} has degree at least $k$ and a \emph{$k$-neighbor} of a vertex
$v$ is a $k$-vertex adjacent to $v$.
Choose $r$ to be a 3-vertex with as few vertices at distance 2 as possible.
For each integer $i$, let $N_i$ denote the set of vertices at distance
$i$ from $r$.  Also, let $N_{4^+}=\bigcup_{i\ge 4}N_i$.
We first handle the case $|N_2|\ge 8$, which is short.

\setcounter{claim}{0}
\begin{claim}
If $|N_2|\ge 8$, then $e(G)\ge \frac53n$.
\end{claim}

Since $r$ was chosen among all 3-vertices to minimize $N_2$, each 3-vertex has
either a $5^+$-neighbor or at least two 4-neighbors.  Thus, we let $\ch(v)=d(v)$
and use the following discharging rule.

\begin{enumerate}
\item Each 3-vertex takes $\frac16$ from each 4-neighbor and $\frac13$
from each $5^+$-neighbor.
\end{enumerate}

If $d(v)\ge 5$, then $\ch^*(v)\ge d(v)-\frac13d(v)=\frac23d(v)\ge\frac{10}3$.
If $d(v)=4$, then $\ch^*(v)\ge d(v)-\frac16d(v)=4-\frac46=\frac{10}3$.
If $d(v)=3$, then $\ch^*(v)\ge 3+\frac13=\frac{10}3$ or $\ch^*(v)\ge 3+\frac26=\frac{10}3$.
Hence, $e(G)=
\frac12\sum_{v\in V(G)}\ch(v)=
\frac12\sum_{v\in V(G)}\ch^*(v)\ge \frac53n$.  This proves the claim.
\bigskip

Hereafter, we assume that $|N_2|\le 7$.  Now a variation on the 
\hyperref[smallNeighborhood]{Small Neighborhood Lemma}
implies that $d(v)\ge 4$ for each
vertex $v\in N_{4^+}$.  Suppose instead that $d(v)= 3$ for some vertex $v\in
N_{4^+}$.  Let $p$ be the configuration with 15 pebbles on $r$, 0 pebbles on
each vertex in $N_1\cup N_2\cup N[v]$, and 1 pebble on each other vertex.  Since
$|\{r\}\cup N_1\cup N_2\cup N[v]|\le 15$, the configuration has at least $n$
pebbles, but no pebble can reach $v$, since at most one pebble can leave
$N[r]\cup N_2$.  This contradicts that $G$ is Class 0.  Thus, $d(v)\ge 4$ for
each $v\in N_{4^+}$.

Now we again redistribute charge.  We let $\ch(v)=d(v)$ and we use the
following two discharging rules.

\begin{enumerate}
\item Each vertex in $N_2$ takes charge 1 from its neighbor in $N_1$.
\item Each vertex in $N_3$ takes charge $\frac13$ from its neighbor in $N_2$.
\end{enumerate}

We show that each vertex in $V(G)\setminus N[r]$ finishes with charge at least
$\frac{10}3$.  If $v\in N_{4^+}$, then $\ch^*(v)=\ch(v)=d(v)\ge 4$.
If $v\in N_3$, then $\ch^*(v)\ge d(v)+\frac13\ge\frac{10}3$.  If $v\in N_2$, then
$\ch^*(v)\ge d(v)+1-\frac13(d(v)-1)=\frac23d(v)+\frac43\ge\frac{10}3$.  The
total charge on vertices of $\{r\}\cup N_1$ is $3+3(1)=6$.  Thus, the sum of all
final charges is at least $\frac{10}3(n-4)+6 = \frac{10}3n-\frac{22}3$.  Thus,
$e(G)\ge \frac53n-\frac{11}3$.
\end{proof}

\subsection{Diameter 2}

We now prove that every $n$-vertex diameter 2 Class 0 graph $G$ has at least $2n-5$
edges.  This bound is best possible.  Before proving this result, we describe
some graphs where equality holds.  In what follows, we show that these are the
\emph{only} graphs where equality holds.
To begin, we need the following lemma.

\begin{lem}
Given a graph $G$ and a vertex $v\in V(G)$, form $G'$ from $G$ by adding a new
vertex, $v'$, with $N(v')=N(v)$.  If $G$ is Class 0, then $G'$ is also Class 0.
\label{lem:clone}
\end{lem}

\begin{proof}
Let $G$ be Class 0, and form $G'$ from $G$ as in the lemma.  We show that
$G'$ is Class 0.  Let $p'$ be a configuration of size $|V(G')|$ on $G'$ and 
$r$ be a target vertex in $G'$.

First suppose that $r\notin \{v,v'\}$.  
We form configuration $p$ for $G$ as follows.  
Let $p(w)=p'(w)$ for all $w\in V(G)\setminus\{v\}$, and let
$p(v)=\max(p'(v)+p'(v')-1,0)$.  Now $|p|\ge |V(G)|$, so $r$ is reachable from
$p$ in $G$; let $\sigma$ be a pebbling sequence that reaches $r$ in $G$. If
$\sigma$ reaches $r$ from $p'$ in $G'$, then we are done.  Otherwise, $v$ must
make more moves in $\sigma$ in $G$ from $p$ than are possible in $G'$ from $p'$.
Now all of these ``extra'' moves from $v$ can be made instead from $v'$ (precisely
because $p(v) = p'(v)+p'(v')-1$).  Thus $r$ is reachable in $G'$ from $p'$.

Suppose instead that $r\in \{v,v'\}$; by symmetry, assume that $r=v$.  We may
assume that $p'(v)=0$ and $p'(v')<4$.  If $p'(v')\le 1$, then we can proceed as
in the previous paragraph.  So assume that $p'(v')\in\{2,3\}$.  Since $G$ is
Class 0, Lemma~\ref{lem:NoCut} implies that $d(v)\ge 2$.  Choose $u_1,u_2\in
N(v)$.  Since $p(v')\in \{2,3\}$, we can assume that $p(u_1)=p(u_2)=0$.  We form
$p$ for $G$ as follows.  Let $p(w)=p'(w)$ for all $w\in
V(G)\setminus\{u_1,u_2\}$ and $p(u_1)=p(u_2)=1$.  Now $|p|\ge |V(G)|$, so $v$ is
reachable from $p$ in $G$; let $\sigma$ be a pebbling sequence that reaches $v$
in $G$.  If $\sigma$ makes no moves from $u_1$ or $u_2$, then $\sigma$ also
reaches $v$ from $p'$ in $G'$.  So assume that $\sigma$ makes a move from $u_1$
or $u_2$.  Form $\sigma'$ from $\sigma$ by truncating $\sigma$ just before the
first time that it moves from $u_1$ or $u_2$, say $u_1$, and then appending a move
from $v'$ to $u_1$ and a move from $u_1$ to $v$.  Now $\sigma'$ reaches $v$ from
$p'$ in $G$.  Thus, $G'$ is Class 0.
\end{proof}

Now we use Lemma~\ref{lem:clone} to show that two infinite families of graphs are
all Class 0.

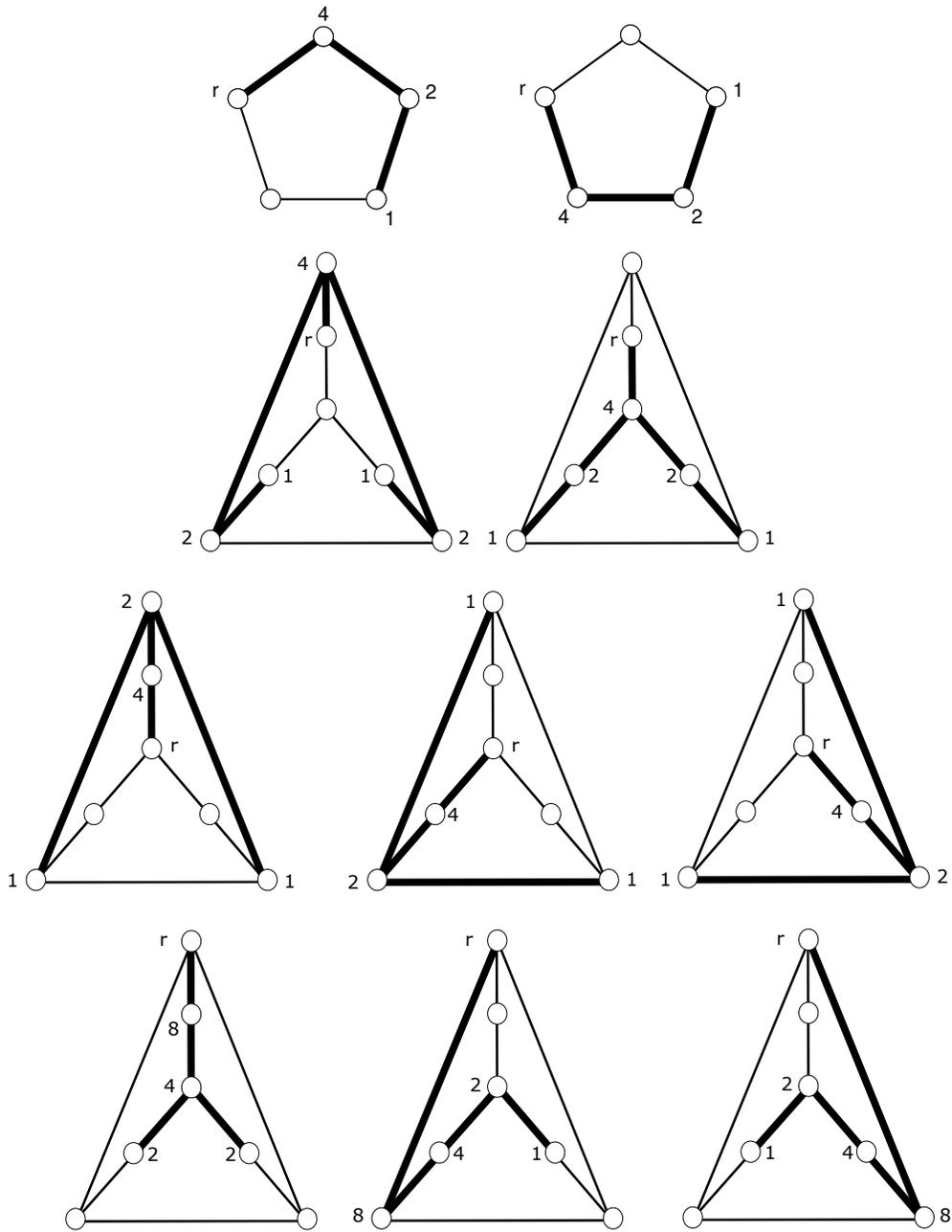
\begin{figure}[!t]
\begin{center}
\renewcommand{\ttdefault}{phv}
\begin{tikzpicture}[rotate=18,scale=.5, font=\sffamily]
\tikzstyle{vertex style}=[shape = circle, minimum size = 7.75pt, inner sep = 0.6pt, draw]
    \begin{scope} 
\fontsize{8pt}{12pt}\selectfont
      \path \foreach \i/\name in {0/{2,4},1/{r},2/{1,4},3/{2,5},4/{1,3}}{%
       (72*\i:2.5) node[vertex style] (b\i) {}
}
       ; 
      \foreach \i/\name in {0/{2},1/{4},2/{r},3/{},4/{1}}{%
      \draw (72*\i:3.15) node[] (c\i) {\scriptsize{\texttt{\name}}};
}
\draw[line width=.35em] (b2.center) -- (b1.center) -- (b0.center) -- (b4.center);
\draw (72:2.5) node[vertex style, fill=white] (d1) {};
\draw (144:2.5) node[vertex style, fill=white] (d2) {};
\draw (216:2.5) node[vertex style, fill=white] (d3) {};
\draw (288:2.5) node[vertex style, fill=white] (d4) {};
\draw (360:2.5) node[vertex style, fill=white] (d1) {};
    \end{scope}
     \begin{scope} [edge style/.style={line width=.8pt}]
       \foreach \i  in {0,...,4}{%
       \pgfmathtruncatemacro{\nextb}{mod(\i+1,5)}
       \draw[edge style] (b\i)--(b\nextb);
       }  
     \end{scope}
\begin{scope}[xshift=3.2in,yshift=-1.02in]
\fontsize{8pt}{12pt}\selectfont
      \path \foreach \i/\name in {0/{2,4},1/{r},2/{1,4},3/{2,5},4/{1,3}}{%
       (72*\i:2.5) node[vertex style] (b\i) {}
}
       ; 
      \foreach \i/\name in {0/{1},1/{ },2/{r},3/{4},4/{2}}{%
      \draw (72*\i:3.15) node[] (c\i) {\scriptsize{\texttt{\name}}};
}
\draw[line width=.35em] (b2.center) -- (b3.center) -- (b4.center) -- (b0.center);
\draw (72:2.5) node[vertex style, fill=white] (d1) {};
\draw (144:2.5) node[vertex style, fill=white] (d2) {};
\draw (216:2.5) node[vertex style, fill=white] (d3) {};
\draw (288:2.5) node[vertex style, fill=white] (d4) {};
\draw (360:2.5) node[vertex style, fill=white] (d1) {};
    \end{scope}
     \begin{scope} [edge style/.style={line width=.8pt}]
       \foreach \i  in {0,...,4}{%
       \pgfmathtruncatemacro{\nextb}{mod(\i+1,5)}
       \draw[edge style] (b\i)--(b\nextb);
       }  
     \end{scope}
  \end{tikzpicture}
~~~~\includegraphics[scale = 0.6]{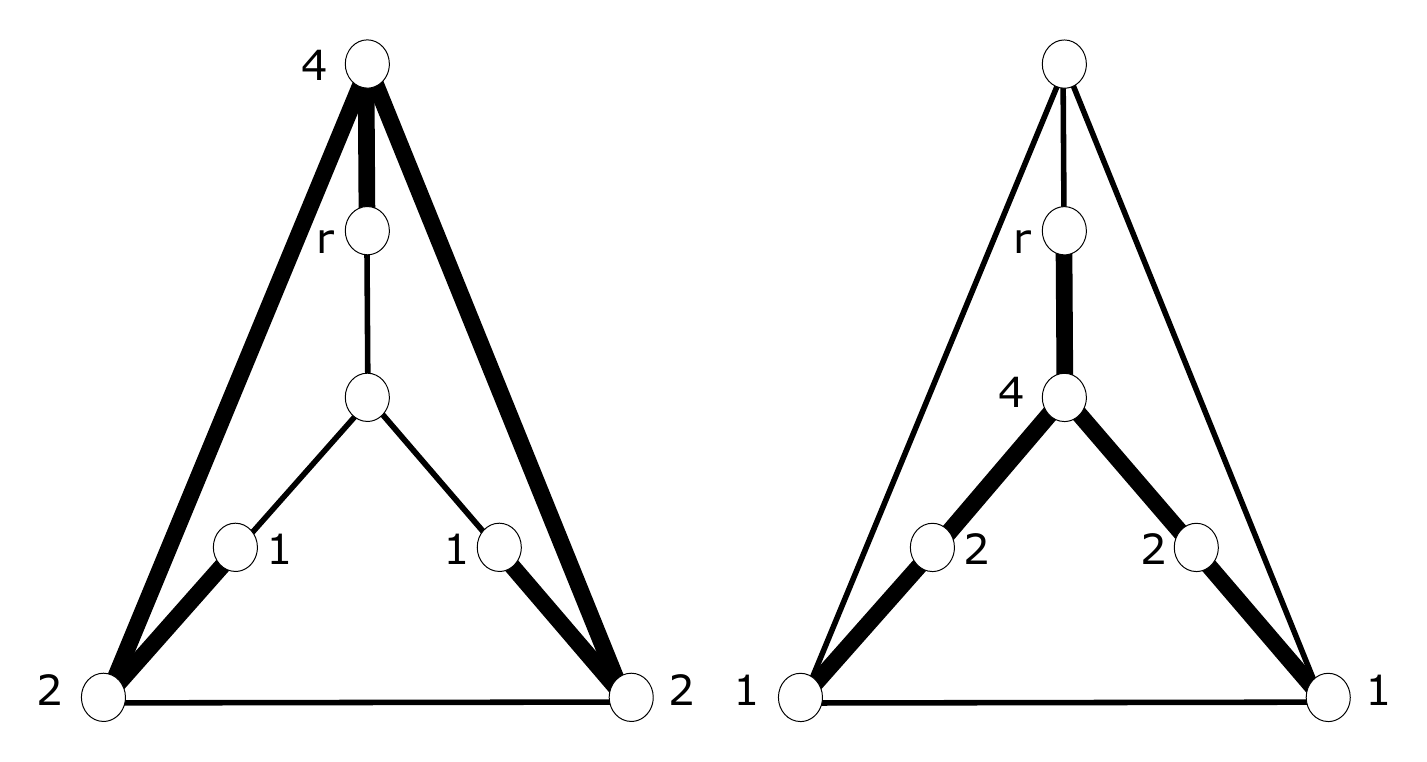}\\
\includegraphics[scale = 0.6]{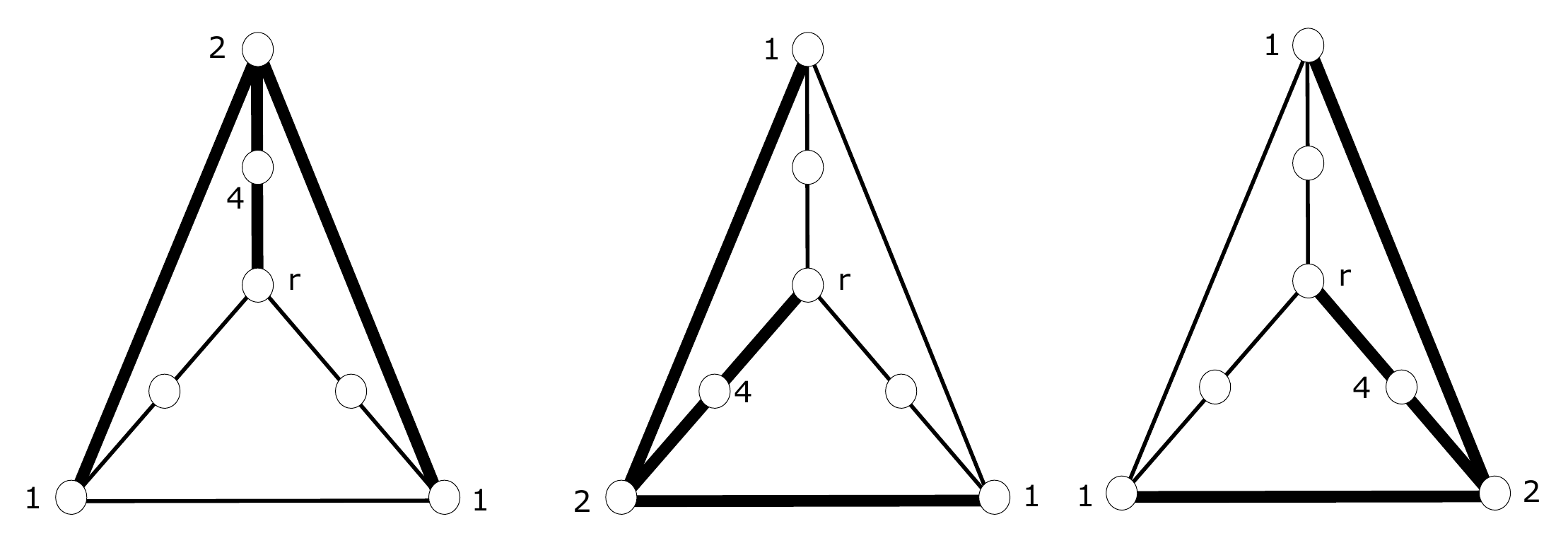}\\
~~~~~~~~\includegraphics[scale = 0.6]{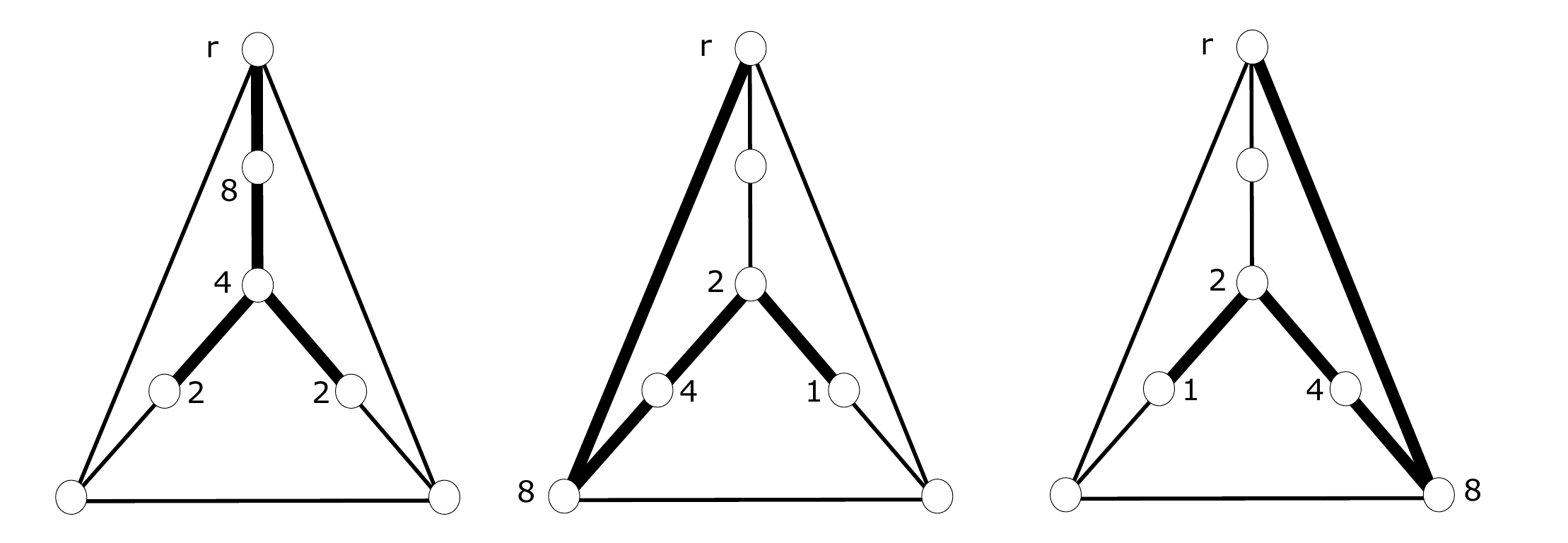}\\
%
\end{center}
\caption{The tree strategies for $C_5$ and for $K_4$ with the edges of a
$K_{1,3}$ subdivided. \label{fig:treeStrat}}
\end{figure}

\begin{example}
The following are two infinite families of Class 0 graphs.  Each $n$-vertex
graph has exactly $2n-5$ edges.
To form an instance of $F_{p,q}$, begin with $K_3$ and replace the two edges
incident to some vertex $v$ with $p$ parallel edges and $q$ parallel edges
(where $p$ and $q$ are positive); finally, subdivide each of these $p+q$ new
edges.
To form an instance of $G_{p,q,r}$, begin with $K_4$ and replace the three edges
incident to some vertex $v$ with $p$ parallel edges, $q$ parallel edges, and
$r$ parallel edges (where $p$, $q$, and $r$ are positive); finally, subdivide
each of these $p+q+r$ new edges.
\end{example}

It is easy to see that each $n$-vertex graph in $F_{p,q}$ has $2n-5$ edges,
since the 2-vertices induce an independent set (when $p\ge 2$ and $q\ge 2$),
and the three high-degree vertices have among them a single edge.
Similarly, $G_{p,q,r}$ has $2n-5$ edges, since the 2-vertices induce an
independent set and the four high-degree vertices have among them 3 edges.  

%
%
We prove that all of $F_{p,q}$ is Class 0, by induction on $p+q$; the induction
step follows immediately from Lemma~\ref{lem:clone}.  The base case is
$F_{1,1}$, which is the 5-cycle.  To show that it is Class 0, we use the tree
strategies shown in the first row of Figure~\ref{fig:treeStrat}.  Since $C_5$ is
vertex-transitive, we can pick the root arbitrarily.  Let $w$ be the sum of the weights in
the two tree strategies for $C_5$.  Note that $w(v)\ge 3$ for every vertex $v\in
V(G)\setminus\{r\}$ and $\sum_{v\in V(G)\setminus \{r\}}w(v) = 14 < 3(4+1)$.
Thus, by the \hyperref[coveringLemma]{Covering Lemma}, $C_5$ is Class 0.

We prove that all of $G_{p,q,r}$ is Class 0, by induction on $p+q+r$; the induction
step follows immediately from Lemma~\ref{lem:clone}.  The base case is
$G_{1,1,1}$.  To show that $G_{1,1,1}$ is Class 0, we use the tree
strategies shown in Figure~\ref{fig:treeStrat}.  Up to symmetry, $G_{1,1,1}$ has
three types of vertices: a degree 2 vertex, the center degree 3 vertex, and a
peripheral degree 3 vertex.  The tree strategies for these cases are given in
the first, second, and third row below the strategies for $C_5$.

Let $r$ be a degree 2 vertex, and let $w(v)$ be the sum of the two weight
functions in the second row of Figure~\ref{fig:treeStrat}.  Note that $w(v)\ge 3$
for all $v\in V(G)\setminus\{r\}$.  Further, $\sum_{v\in V(G)\setminus
\{r\}}w(v)=20<3(6+1)$.  Thus, the \hyperref[coveringLemma]{Covering Lemma}
implies that $\pi(G_{1,1,1},r)\le |V(G_{1,1,1})|$.
Now let $r$ be the center vertex, and let $w(v)$ be the sum of the three weight
functions in the third row of Figure~\ref{fig:treeStrat}.  Note that $w(v)= 4$
for all $v\in V(G)\setminus\{r\}$.  Thus, $\sum_{v\in V(G)\setminus
\{r\}}w(v)=24<4(6+1)$.  Thus, the \hyperref[coveringLemma]{Covering Lemma}
implies that $\pi(G_{1,1,1},r)\le |V(G_{1,1,1})|$.
Finally, let $r$ be a peripheral vertex, and let $w(v)$ be the sum of the three weight
functions in the fourth row of Figure~\ref{fig:treeStrat}.  Note that $w(v)\ge 7$
for all $v\in V(G)\setminus\{r\}$.  Further, $\sum_{v\in V(G)\setminus
\{r\}}w(v)=46<7(6+1)$.  Thus, the \hyperref[coveringLemma]{Covering Lemma}
implies that $\pi(G_{1,1,1},r)\le |V(G_{1,1,1})|$.  Since $\pi(G_{1,1,1},r)\le
|V(G_{1,1,1})|$ for each root $r$, we conclude that
$\pi(G_{1,1,1})\le|V(G_{1,1,1})|$.  So, $G_{1,1,1}$ is Class 0.
\bigskip

Now we show that every diameter 2 Class 0 graph $G$ has $e(G)\ge 2n-5$ and
characterize when equality holds. Clarke et al.~\cite[Theorem 2.4]{clarke} characterized
diameter 2 graphs that are not Class 0.  It seems likely that we could derive
our result from theirs.  However, we prefer the proof below, since it seems
simpler and more straightforward.  Further, the proof below generalizes to
diameter 2 graphs with no cut-vertices.

\begin{thm}
Let $G$ be an $n$-vertex graph with diameter 2.
If $G$ has no cut-vertex (in particular, if $G$ is Class 0) then $e(G) \ge 2n - 5$.
Further, equality holds if and only if $G$ is the Petersen graph or one of the
graphs in Example 1.
\end{thm}

\begin{proof}
If $\delta(G) \ge 4$, then $e(G) \ge \frac{4n}{2} = 2n$ and
the theorem is true. So we assume $\delta(G) \le 3.$
Lemma~\ref{lem:NoCut} implies that $\delta(G) \ge 2$, so $e(G)
\ge \frac{n\delta(G)}{2} \ge n$. If $n\le 5$, then $e(G)\ge n\ge
2n-5$, so the theorem is true. Thus, we assume $n \ge 6$.
We consider two cases: (i) $\delta(G)=3$ and (ii) $\delta(G)=2$.
%

\textbf{Case 1: $\delta(G) = 3.$}
Choose $r \in V(G)$ with $d(r) = 3$, and let $S=N(r)$.
Each vertex $v\in V(G)\setminus S$ has a neighbor in $S$ and $r$ has 3
neighbors in $S$, so $\sum_{v\in S}d(v)\ge (n-4)+3=n-1$.  Also, $\sum_{v\in
V(G)\setminus S}d(v)\ge \sum_{v\in V(G)\setminus S}3= 3(n-3)$. So $e(G) =
\frac12\sum_{v\in V(G)}d(v)\ge \frac12((n-1)+3(n-3))=\frac12(4n-10)=2n-5$.

If equality holds in $e(G)\ge 2n-5$, then each vertex in $V\setminus S$ has
degree 3 and each
vertex in $V\setminus (S\cup \{r\})$ has exactly one neighbor in $S$.  Let
$\{v_1,v_2,v_3\}=S$, let $S_i=N(v_i)-r$ for each $i\in\{1,2,3\}$,
and let $H=G[S_1\cup S_2\cup S_3]$.
Note that $H$ is a disjoint union of cycles, since each vertex has degree
3 and has exactly one neighbor in $S$.  Also $|S_i|\ge 2$ for each $i$, since
$\delta(G)=3$.  Suppose that $|S_1|\ge 3$, and choose $v\in S_3$. Now $|S_1\cup
S_2|\ge 3+2=5$, so $v_3$ is distance at least 3 from some vertex of $S_1\cup
S_2$ (precisely because $H$ is a disjoint union of cycles).  Hence $|S_1|=2$
and, by symmetry, $|S_2|=|S_3|=2$.  Similarly, if $H$ consists of two 3-cycles,
then some pair of its vertices is distance at least 3 apart.  Hence, $H$
is a 6-cycle.  Further, each pair of vertices in the same $S_i$ are distance 3
apart in $H$.  Thus, if $e(G)=2n-5$, then $G$ is the Petersen graph.

\textbf{Case 2: $\delta (G) = 2$. }
Choose $r\in V(G)$ with $d(r) = 2$, and let $\{v_1,v_2\}=N(r)$.
We partition $V(G)\setminus N(r)$ into three sets, $S_1$, $S_2$, and $S_{1,2}$.
 (Note that $r\in S_{1,2}$.) Let $S_1$ consist of all vertices adjacent only to
$v_1$, $S_2$ of all vertices adjacent only to $v_2$, and $S_{1,2}$ of all
vertices adjacent to both $v_1$ and $v_2$.
Let $H$ be the subgraph induced by $S_1\cup S_2$, and let $H_1,\ldots, H_t$ be
the components of $H$.  We first show that $e(G)\ge 2n-4$ if every $H_i$
either contains a cycle or has a vertex adjacent to some vertex of
$S_{1,2}$.

We assign each edge to one of its endpoints as follows, so that
each vertex other than $v_1$ and $v_2$ has at least 2 assigned edges.
Each edge with exactly one endpoint in $\{v_1,v_2\}$ is assigned to its other
endpoint.  Thus, each vertex of $S_{1,2}$ has 2 assigned edges and
each vertex of $S_1\cup S_2$ has 1 assigned edge.
Suppose that $T$ is some tree component of $H$ and $t$ is a vertex of $T$ with a
neighbor in $S_{1,2}$.  We can direct the edges of $T$ so that $t$ has
outdegree 0 and each other
vertex has outdegree 1.  Now we assign to $t$ its edge to $S_{1,2}$ and assign
to each other vertex of $T$ its out-edge.  When $H_i$ is a component of $H$ with a
cycle, the process is similar.  We choose some spanning tree $T$ of $H_i$ and
choose $t$ to be some vertex incident to an edge of $H_i$ not in $T$.

So assume that some component $H_1$ is a tree and has no neighbor in $S_{1,2}$.
If $V(H_1) \subseteq S_1$, then $v_1$ is a cut-vertex, which is forbidden.
Hence, $V(H_1)\not\subseteq S_1$; similarly, $V(H_1)\not\subseteq S_2$.
Suppose that $H$ has another component, with some vertex $w$.  By
symmetry, assume that $w\in S_1$.  Since $H_1$ has vertices in both $S_1$ and
$S_2$, $w$ is distance at least 3 from some vertex of $H_1$ in $S_2$, a
contradiction.  Thus, $H_1$ is the only component of $H$; so from now on, we say
$H$ for $H_1$.
Choose $t$ arbitrarily in $H$, and direct $E(H)$ and assign edges as above.
Now $t$ has 1 assigned edge and each other vertex has 2 assigned edges,
so $e(G)\ge 2n-5$.
If $v_1v_2\in E(G)$, then $e(G)\ge 2n-4$, so we assume $v_1v_2\notin E(G)$.
Similarly, if any edge has both endpoints in $S_{1,2}$, then $e(G)\ge 2n-4$,
so we assume that $S_{1,2}$ induces an independent set.
In what follows, we characterize when equality holds in $e(G)\ge 2n-5$.

First suppose that $H$ has
leaves in both $S_1$ and $S_2$; call these $u_1$ and $u_2$, respectively.  If
$u_1$ and $u_2$ are adjacent, then $H$ is a single edge,
which is possible; this is $F_{1,q}$, where $q=|S_{1,2}|$.
Now suppose that $u_1$ and $u_2$ are nonadjacent.
Since $G$ is diameter 2, $u_1$ and $u_2$ have some common neighbor, $u_3$.
By symmetry, assume that $u_3\in S_1$.  
Now $u_1$ and $v_2$
must have a common neighbor, so $v_1v_2$ is an edge.  However, now $e(G)\ge 2n-4$.

So assume instead that $H$ has leaves only in one of $S_1$ and $S_2$; by
symmetry, say $S_2$. 
Let $S'_1$ denote the vertices of $S_1$ adjacent to a leaf.  Now $S'_1$ induces
a graph with diameter at most 1 (otherwise some leaf is distance at least 3 from
some vertex of $S'_1$).  Since $H$ is acyclic, $|S'_1|\le 2$.

First suppose that $|S'_1|=1$.  Let $\{u_1\}=S'_1$.  Now all vertices in $S_2$
are leaves of $H$, since $H$ is acyclic.  Further, $u_1$ is the only vertex in
$S_1$.  Thus, $H$ is a star centered at $u_1$.  This is possible; 
$G=F_{p,q}$, where $q=|S_{1,2}|$ and $p$ is the number of leaves of $H$.  

Suppose instead that $|S'_1|=2$, and let $\{u_1,u_2\}=S'_1$.  Now $u_1$ and
$u_2$ are adjacent, since $G$ has diameter 2; otherwise some leaf in $S_2$ is
distance at least 3 from $u_1$ or $u_2$.  Again, each vertex $u_3\in S_2$ must be a
leaf, since $G$ is acyclic.  Finally, $S_1=\{u_1,u_2\}$, again since $G$ has
diameter 2.  Thus, $H$ is a double star,
centered at $u_1$ and $u_2$, with all leaves in $S_2$.
This is also possible; $G=G_{p,q,r}$, where $p=|S_{1,2}|$ and $q$ and $r$
are (respectively) the numbers of leaves of $H$ adjacent to $u_1$ and $u_2$.

%
Hence,
$e(G)=2n-5$ implies that $H$ is (i) a single edge, which is $F_{1,q}$, (b) a star with
its center in $S_1$ (by symmetry) and all of its leaves in $S_2$, which is $F_{p,q}$, or
(c) a double star with both of its centers in $S_1$ and all of its leaves in
$S_2$, which is $G_{p,q,r}$.  This finishes the characterization of when $e(G)=2n-5$.
\end{proof}
\bigskip


\end{document}